\documentclass[12pt]{amsart}
\usepackage{amssymb,amsmath,amsthm}
\usepackage{tikz}
\usepackage{cancel}
\usepackage{mathtools,float}
\usepackage[T1]{fontenc}

\usetikzlibrary{arrows,shapes,automata,backgrounds,decorations,petri,positioning,patterns}

\theoremstyle{plain}
\newtheorem{thm}{Theorem}[section]

\newtheorem{cor}[thm]{Corollary}

\theoremstyle{definition}
\newtheorem{defn}[thm]{Definition}

\title{On the expected number of commutations in reduced words}

\author[Bridget Eileen Tenner]{Bridget Eileen Tenner$^{\dagger}$}
\address{Department of Mathematical Sciences, DePaul University, Chicago, IL 60614}
\email{bridget@math.depaul.edu}

\thanks{$^{\dagger}$ Research partially supported by a Simons Foundation Collaboration Grant for Mathematicians.}


\begin{document}

\begin{abstract}
We compute the expected number of commutations appearing in a reduced word for the longest element in the symmetric group.  The asymptotic behavior of this value is analyzed and shown to approach the length of the permutation, meaning that nearly all positions in the reduced word are expected to support commutations.  Finally, we calculate the asymptotic frequencies of commutations, consecutive noncommuting pairs, and long braid moves.\\

\end{abstract}

\maketitle

\section{Introduction}

The symmetric group $\mathfrak{S}_n$ is generated by the simple reflections $\{s_1,\ldots, s_{n-1}\}$, where $s_i$ interchanges $i$ and $i+1$. These reflections satisfy the Coxeter relations
\begin{eqnarray}
\nonumber &s_i^2=1& \text{ for all $i$,}\\
\label{eqn:short} &s_is_j = s_js_i& \text{ for $|i - j| > 1$, and}\\
\label{eqn:long} &s_is_{i+1}s_i = s_{i+1}s_is_{i+1}.&
\end{eqnarray}

\begin{defn}
Relations as in equation~\eqref{eqn:short} are called \emph{commutations} (also, \emph{short braid relations}), and those as in equation~\eqref{eqn:long} are \emph{long braid relations} (also, \emph{Yang-baxter moves}).
\end{defn}

There is a well-defined notion of ``length'' for a permutation, and it can be defined in terms of these simple reflections $\{s_i\}$.

\begin{defn}
Consider $w \in \mathfrak{S}_n$ written as a product $w = s_{i_1}s_{i_2}\cdots s_{i_{\ell(w)}}$ for $1 \le i_j \le n-1$, where $\ell(w)$ is minimal.  This $\ell(w)$ is the \emph{length} of $w$, while the product $s_{i_1}s_{i_2}\cdots s_{i_l}$ is a \emph{reduced decomposition} for $w$ and the string of subscripts $i_1i_2\cdots i_l$ is a \emph{reduced word} for $w$.
\end{defn}

The Coxeter relations \eqref{eqn:short} and \eqref{eqn:long} have obvious analogues in reduced words.  Namely, the letters $\{i,j\}$ commute when $|i - j| > 1$, and $i(i+1)i$ can be rewritten as $(i+1)i(i+1)$.  We will abuse terminology slightly and refer to these phenomena in reduced words as commutations and long braid moves, respectively, as well.

The structure of reduced words, and the influence of commutations and long braid moves, is of great interest.  Some facts are known, such as those cited below, but many open questions remain. Additionally, there are objects defined naturally in terms of reduced words and Coxeter relations, including commutation classes and the graph of those classes (see, for example, \cite{elnitsky} and \cite{tenner rdpp}), for which much of the architecture is still largely obscure.

\begin{defn}
There is a \emph{longest element} $w_0 \in \mathfrak{S}_n$, and $\ell(w_0) = \binom{n}{2}$.  In one-line notation, this $w_0$ is the permutation $n(n-1)(n-2)\cdots 321$.
\end{defn}

In \cite{reiner}, Reiner showed the following surprising fact about long braid moves in reduced words for $w_0 \in \mathfrak{S}_n$.

\begin{thm}[{\cite[Theorem 1]{reiner}}]\label{thm:reiner}
For all $n \ge 3$, the expected number of long braid moves in a reduced word for $w_0 \in \mathfrak{S}_n$ is $1$.
\end{thm}

Similar calculations were made for the finite Coxeter group of type $B$ in \cite{tenner expb}, one of which is restated here.

\begin{thm}[{\cite[Theorem 3.1]{tenner expb}}]\label{thm:expb}
For all $n \ge 3$, the expected number of long braid moves in a reduced word for the longest element in the Coxeter group of type $B_n$ is $2 - 4/n$.
\end{thm}

Recent work on random reduced words for the longest element appears under the guise of ``random sorting networks,'' as in \cite{ahrv, ah, agh, warrington}.

In the present article, we compute the expected number of commutations in a reduced word for the longest element $w_0 \in \mathfrak{S}_n$.  We obtain an explicit sum for this expectation, which is stated in Theorem~\ref{thm:main}.  Unfortunately,  occurrences of commutations are not as tidy as occurrences of long braid moves, and so our result is not independent of $n$ as in Reiner's work \cite{reiner}, nor of such straightforward form as the results of \cite{tenner expb}.

In the final section of the paper, we discuss the asymptotic behavior of the enumeration from Theorem~\ref{thm:main}.  Finally, we use that analysis to produce Corollary~\ref{cor:proportions}, giving the asymptotic frequencies of commutations, consecutive noncommuting pairs, and long braid moves in reduced words for the longest element.

\section{Enumerating commutations}

Throughout this section, fix a positive integer $n$ and consider $w_0 \in \mathfrak{S}_n$.  Let $\ell = \ell(w_0) = \binom{n}{2}$.

\begin{defn}
For any integer $k \in [1,\ell-1]$, say that a reduced word $i_1 \cdots i_{\ell}$ for $w_0$ \emph{supports a commutation in position $k$} if
$$|i_k - i_{k+1}| > 1.$$
\end{defn}

We want to compute the expected number of commutations in a reduced word for the longest element $w_0 \in \mathfrak{S}_n$.

In particular, we are interested in the following statistic.

\begin{defn}
Let $C_n$ be the random variable on a reduced word for $w_0 \in \mathfrak{S}_n$, which counts positions in the word that support commutations.
\end{defn}

To analyze $C_n$, we will actually look at the complementary event; that is, consecutive noncommuting pairs of symbols in a reduced word.

\begin{defn}
Consider a reduced word $i_1i_2\cdots i_{\ell}$ for $w_0$.  Let $A_n^{(k,j)}$ be the indicator random variable for the event that $i_k = j$ and $i_{k+1} = j+1$, and let
$$A_n = \sum_{k=1}^{\ell-1}\sum_{j=1}^{n-2}A_n^{(k,j)}.$$
Similarly, let $B_n^{(k,j)}$ be the indicator random variable for the event that $i_k = j+1$ and $i_{k+1} = j$, and let
$$B_n = \sum_{k=1}^{\ell-1}\sum_{j=1}^{n-2}B_n^{(k,j)}.$$
\end{defn}

Therefore $E(A_n)+E(B_n)$ computes the expected number of consecutive noncommuting pairs in a reduced word for $w_0$, and so
\begin{equation}\label{eqn:complement}
E(C_n) = \ell-1 - \big(E(A_n) + E(B_n)\big).
\end{equation}
In fact, the symmetry $s_{n-1-i}w_0s_i = w_0$ gives a correspondence between $A_n^{(k,j)}$ and $B_n^{(k,n-1-j)}$, and so we can simplify equation~\eqref{eqn:complement} to
\begin{equation}\label{eqn:complement2}
E(C_n) = \ell-1 - 2E(A_n).
\end{equation}

The proof of Theorem~\ref{thm:main} will employ similar techniques to those used by Reiner in \cite{reiner}.  In particular, we will recognize that particular permutations are vexillary, and then use a result of Stanley to enumerate reduced words.

\begin{defn}
A permutation is \emph{vexillary} if it avoids the pattern $2143$.  That is, $w(1)\cdots w(n)$ is vexillary if there are no indices $1 \le i_1 < i_2 < i_3 < i_4 \le n$ such that $w(i_2) < w(i_1) < w(i_4) < w(i_3)$.
\end{defn}

Stanley showed in \cite{stanley} that the reduced words for a vexillary permutation are enumerated by standard Young tableaux of a particular shape, which can be enumerated using the hook-length formula, as discussed in \cite{stanley ec2}.

\begin{defn}
Fix a permutation $w = w(1)\cdots w(n)$.  For all integers $i \in [1,n]$, let
$$r_i = \left|\{j : j < i \text{ and } w(j) > w(i)\}\right|.$$
The partition $\lambda(w) = (\lambda_1(w),\lambda_2(w),\ldots)$ is defined by writing $\{r_1,\ldots,r_n\}$ in nonincreasing order. 
\end{defn}

\begin{thm}[{\cite[Corollary 4.2]{stanley}}]\label{thm:stanley}
If $w$ is vexillary, then the number of reduced words for $w$ is equal to $f^{\lambda(w)}$, the number of standard Young tableaux of shape $\lambda(w)$.
\end{thm}

We are now able to compute the expected number of commutations in a reduced word for $w_0 \in \mathfrak{S}_n$.

\begin{thm}\label{thm:main}
For all $n \ge 3$,
$$E(C_n) =
\binom{n}{2} - 1 - \frac{1}{3\binom{n}{2} 2^{2n-7}} \sum_{j=1}^{n-2} \frac{(2j-1)!!}{(j-1)!}\cdot \frac{(2j+1)!!}{j!}\cdot \frac{(2n-2j-3)!!}{(n-j-2)!}\cdot \frac{(2n-2j-1)!!}{(n-j-1)!}.$$
\end{thm}

\begin{proof}
Fix an integer $n \ge 3$ and let $\ell = \binom{n}{2}$.  To calculate the expected number of positions that support a commutation in a reduced word for $w_0 \in \mathfrak{S}_n$, we will analyze the complementary event and use equation~\eqref{eqn:complement2}.

As noted in \cite{reiner}, the fact that $s_iw_0s_{n-i} = w_0$ for all $i \in [1,n-1]$ gives a cyclic symmetry to reduced words for $w_0$:
$$i_1i_2\cdots i_{\ell}$$
is a reduced word for $w_0$ if and only if
$$i_2i_3 \cdots i_{\ell}(n-i_1)$$
is as well.  Therefore it is enough to consider the appearance of a commutation (or not) in the first position of a reduced word, using the identity
$$A_n^{(k,j)} = A_n^{(1,j)}$$
for all $1 \le k \le \ell-1$.  Thus equation~\eqref{eqn:complement2} can be rewritten as
$$E(C_n) = \ell - 1 - 2(\ell - 1)\sum_{j=1}^{n-2} E(A_n^{(1,j)}).$$

The expected value $E(A_n^{(1,j)})$ is equal to the probability that the reduced word for $w_0$ is
$$j(j+1)i_3i_4\cdots i_{\ell},$$
in which case $i_3i_4\cdots i_{\ell}$ is a reduced word for the permutation
\begin{eqnarray*}
a_n^{(j)} &=& s_{j+1}s_jw_0\\
&=& n(n-1)(n-2)\cdots(j+3)(j+1)j(j+2)(j-1)\cdots321,
\end{eqnarray*}
in one-line notation.

The permutation $a_n^{(j)}$ is vexillary, meaning that we can enumerate its reduced words by counting standard Young tableaux as described in Theorem~\ref{thm:stanley}.  The shape $\lambda(a_n^{(j)})$ is obtained from the staircase shape $\delta_n$ by deleting the corner squares from rows $j$ and $j+1$ as in Figure~\ref{fig:shapes}.  Therefore,
\begin{equation}\label{eqn:expectation in terms of shapes}
E(C_n) = \ell - 1 - 2(\ell - 1)\sum_{j=1}^{n-2} \frac{f^{\lambda(a_n^{(j)})}}{f^{\delta_n}}.
\end{equation}

\begin{figure}[htbp]
\begin{tikzpicture}[scale=.5]
\fill[black!20] (0,5) rectangle (5,6);
\fill[black!20] (0,6) rectangle (6,7);
\fill[black!20] (4,7) rectangle (6,9);
\foreach \x in {1,2,3,4,5,6,7,8} {\draw (0,\x) --+ (\x,0); \draw(\x,9) --+ (0,\x-9);}
\draw (0,1) -- (0,9) -- (8,9);
\draw (5,1) node {$\delta_9$};
\draw (6,1) node {\phantom{$\lambda(a_9^{(3)})$}};
\end{tikzpicture}
\hspace{.5in}
\begin{tikzpicture}[scale=.5]
\fill[black!20] (0,5) rectangle (4,6);
\fill[black!20] (0,6) rectangle (5,7);
\fill[black!20] (4,7) rectangle (6,9);
\foreach \x in {1,2,3,4,5,6,7,8} {\draw (0,\x) --+ (\x,0); \draw(\x,9) --+ (0,\x-9);}
\draw (0,1) -- (0,9) -- (8,9);
\draw (5,1) node {$\lambda(a_9^{(3)})$};
\fill[white] (4.05,4.8) rectangle (5.2,5.95);
\fill[white] (5.05,5.8) rectangle (6.2,6.95);
\foreach \x in {(4.5,5.5),(5.5,6.5)} {\fill[black] \x circle (4pt);}
\end{tikzpicture}
\caption{The staircase shape $\delta_9$ corresponding to $w_0 \in \mathfrak{S}_9$ and the shape $\lambda(a_9^{(3)})$. The shaded cells are the only ones in which the hook-lengths for the two shapes will differ.}
\label{fig:shapes}
\end{figure}
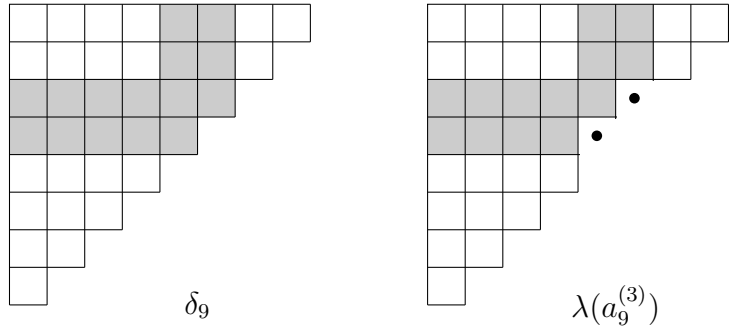

We evaluate equation~\eqref{eqn:expectation in terms of shapes} using the hook-length formula.  In fact, many hook-lengths of $\delta_n$ and $\lambda(a_n^{(j)})$ are the same, and thus cancel in the ratio
$$\frac{f^{\lambda(a_n^{(j)})}}{f^{\delta_n}}.$$
The only non-identical hook-lengths in the shapes $\delta_n$ and $\lambda(a_n^{(j)})$ occur as indicated in Figure~\ref{fig:shapes}, and so equation~\eqref{eqn:expectation in terms of shapes} becomes
$$E(C_n) = 
\ell - 1 - 2(\ell - 1)\sum_{j=1}^{n-2} \frac{(\ell-2)!}{\ell!}\cdot3 \cdot \left(\frac{2}{3}\right)^2 u_{j-1}  u_j  u_{n-j-2}  u_{n-j-1},$$
where
$$u_i = \frac{3\cdot5\cdots(2i+1)}{2\cdot4\cdots(2i)}.$$
Thus
\begin{eqnarray*}
E(C_n) &=& \ell - 1 - \frac{8}{3\ell}\sum_{j=1}^{n-2} u_{j-1}u_ju_{n-j-2}u_{n-j-1}\\
&=&\ell - 1 - \frac{1}{3\ell 2^{2n-7}} \sum_{j=1}^{n-2} \frac{(2j-1)!!}{(j-1)!}\cdot \frac{(2j+1)!!}{j!}\cdot \frac{(2n-2j-3)!!}{(n-j-2)!}\cdot \frac{(2n-2j-1)!!}{(n-j-1)!},
\end{eqnarray*}
completing the proof.
\end{proof}

\section{Asymptotic behavior}

In this last section of the paper, we analyze the expected value $E(C_n)$ as $n$ goes to infinity.  In fact, we will draw conclusions about appearances of all three possibilities of consecutive symbols in a reduced word for the longest element: commutations, consecutive noncommuting pairs $i(i+1)$ or $(i+1)i$, and long braid moves.

Set
$$\sigma_n^{(j)} = \frac{1}{3\binom{n}{2} 2^{2n-7}} \cdot \frac{(2j-1)!!}{(j-1)!}\cdot \frac{(2j+1)!!}{j!}\cdot \frac{(2n-2j-3)!!}{(n-j-2)!}\cdot \frac{(2n-2j-1)!!}{(n-j-1)!},$$
and recall from Theorem~\ref{thm:main} that
$$\sum_{j=1}^{n-2} \sigma_n^{(j)} = 2E(A_n).$$
In particular, this is the expected number of consecutive noncommuting pairs in a reduced word for $w_0 \in \mathfrak{S}_n$.

Central binomial coefficients satisfy
$$\binom{2x}{x} \le \frac{4^x}{\sqrt{\pi x}}.$$
Combining this with the fact that
$$\frac{(2x-1)!!}{(x-1)!} = \binom{2x}{x}\cdot \frac{x}{2^x}$$
yields
\begin{eqnarray*}
\sigma_n^{(j)} &=& \frac{256}{3\pi^2} \cdot \frac{\sqrt{j(j+1)(n-j-1)(n-j)}}{n(n-1)}\\
&\approx& \frac{256}{\ 3\pi^2} \cdot\frac{j(n-j)}{n^2}.
\end{eqnarray*}
Similarly, 
$$\binom{2x}{x} \ge \frac{4^x}{\sqrt{\pi x}}\cdot \left(1 - \frac{1}{8x}\right)$$
yielding the lower bound
\begin{eqnarray*}
\sigma_n^{(j)} &\ge& \frac{256}{3\pi^2} \cdot \frac{\sqrt{j(j+1)(n-j-1)(n-j)}}{n(n-1)}\\
& & \hspace{.25in} \cdot \left(1 - \frac{1}{8j}\right) \cdot \left(1 - \frac{1}{8(j+1)}\right) \cdot \left(1 - \frac{1}{8(n-j-1)}\right) \cdot \left(1 - \frac{1}{8(n-j)}\right)\\
&\approx& \frac{256}{\ 3\pi^2} \cdot\frac{j(n-j)}{n^2} \left(1 - \frac{n}{8j(n-j)}\right)^2.
\end{eqnarray*}
Therefore
$$\sigma_n^{(j)} = \frac{256}{\ 3\pi^2} \cdot\frac{j(n-j)}{n^2} + O\left(\frac{1}{n}\right).$$

From this we can compute the asymptotic behavior of the expected number of consecutive pairs of noncommuting symbols in a reduced word for $w_0 \in \mathfrak{S}_n$:
\begin{eqnarray}
\nonumber 2E(A_n) &=& \sum_{j=1}^{n-2}\sigma_n^{(j)}\\
\nonumber &=& \frac{256}{3\pi^2}\sum_{j=1}^{n-2}\frac{j(n-j)}{n^2} \ + \ O(1)\\
\nonumber &=& \frac{256}{\ 3\pi^2} \left(\frac{(n-1)(n-2)}{2n} - \frac{(n-1)(n-2)(2n-3)}{6n^2}\right) + O(1)\\
\nonumber &=& \frac{128}{\ 9\pi^2}\cdot \frac{(n-1)(n-2)(n+3)}{n^2} + O(1)\\
\label{eqn:asymp noncomm}&=& \frac{128}{\ 9\pi^2}n + O(1).
\end{eqnarray}

Combining equation~\eqref{eqn:asymp noncomm} with Theorem~\ref{thm:main} yields
\begin{eqnarray}
\nonumber E(C_n) &=& \binom{n}{2} - 1 - 2E(A_n)\\
\label{eqn:asymp comm} &=& \binom{n}{2} - \frac{128}{\ 9\pi^2}n + O(1).
\end{eqnarray}

It is interesting to compare equation~\eqref{eqn:asymp comm} with Theorem~\ref{thm:reiner}.  These computations imply that appearances of commutations and appearances of long braid moves have vastly different behaviors in reduced words for the longest element.  Namely, while only one long braid move is expected to appear, nearly all positions are expected to support commutations.  This is, perhaps, not surprising because commutations are positions in a reduced word where a symbol $i$ is followed by anything other than $\{i-1,i,i+1\}$, and the proportion $(n-4)/(n-1)$ of ``acceptable'' successors to $i$ approaches $1$ as $n$ gets large.

To conclude this article, we use equations~\eqref{eqn:asymp noncomm} and~\eqref{eqn:asymp comm} and Theorem~\ref{thm:reiner} to state the asymptotic expectation of the proportion of appearances of Coxeter-related behaviors in reduced words for $w_0 \in \mathfrak{S}_n$.

\begin{cor}\label{cor:proportions}
Consider the reduced words for the longest element $w_0 \in \mathfrak{S}_n$.
\begin{itemize}
\vspace{.1in}
\item $\displaystyle{\frac{E(\text{number of commutations})}{\text{length}} \ \approx \ 1}$
\vspace{.1in}
\item $\displaystyle{\frac{E(\text{number of consecutive noncommuting pairs})}{\text{length}} \ \approx \ \frac{256}{9\pi^2n}}$
\vspace{.1in}
\item $\displaystyle{\frac{E(\text{number of long braid moves})}{\text{length}} \ \approx \ \frac{2}{n^2}}$
\end{itemize}
\end{cor}

\section*{Acknowledgements}

I am indebted to Richard Stanley for helpful comments and enthusiasm about these results, and to Andrew Rechnitzer for asymptotic insight and Sage advice. I also appreciate the suggestions of the anonymous referees.


\begin{thebibliography}{99}

\bibitem{ahrv} O.~Angel, A.~E.~Holroyd, D.~Romik, and B.~Virag, Random sorting networks, \textit{Adv.~Math.} \textbf{215} (2007), 839--868.

\bibitem{ah} O.~Angel and A.~E.~Holroyd, Random subnetworks of random sorting networks, \textit{Electron.~J.~Combin.} \textbf{17} (2010), N23.

\bibitem{agh} O.~Angel, V.~Gorin, and A.~E.~Holroyd, A pattern theorem for random sorting networks, \textit{Electron.~J.~Probab.} \textbf{17} (2012), 1--16.

\bibitem{elnitsky} S.~Elnitsky, Rhombic tilings of polygons and classes of reduced words in Coxeter groups, \textit{J.~Combin.~Theory, Ser.~A} \textbf{77} (1997), 193--221.

\bibitem{reiner} V.~Reiner, Note on the expected number of Yang-Baxter
moves applicable to reduced decompositions, \textit{European J.~Combin.} \textbf{26} (2005), 1019--1021.

\bibitem{stanley ec2} R.~P.~Stanley, \textit{Enumerative Combinatorics}, Vol.~2, Cambridge Studies in Advanced Mathematics, vol.~62, Cambridge University Press, Cambridge, 1999.

\bibitem{stanley} R.~P.~Stanley, On the number of reduced decompositions of elements of Coxeter groups, \textit{European J.~Combin.} \textbf{5} (1984), 359--372.

\bibitem{tenner expb} B.~E.~Tenner, On expected factors in reduced decompositions in type $B$, \textit{European J.~Combin.} \textbf{28} (2007), 1144--1151.

\bibitem{tenner rdpp} B.~E.~Tenner, Reduced decompositions and permutation patterns, \textit{J.~Algebr.~Comb.} \textbf{24} (2006), 263--284.

\bibitem{warrington} G.~Warrington, A combinatorial version of Sylvester's four-point problem, \textit{Adv.~Appl.~Math.} \textbf{45} (2010), 390--394.

\end{thebibliography}
\end{document}